\documentclass[reqno, 12pt,dvipsnames ]{amsart}

\usepackage[a4paper,margin=1.15in]{geometry}
\usepackage{amsmath,amssymb,amsthm,mathtools}
\usepackage{enumitem}
\usepackage{microtype}

\usepackage{xcolor}

\usepackage[colorlinks=true,
            linkcolor=black,
            citecolor=blue,
            urlcolor=blue]{hyperref}

\newtheorem{theorem}{Theorem}[section]
\newtheorem{proposition}[theorem]{Proposition}
\newtheorem{lemma}[theorem]{Lemma}
\newtheorem{corollary}[theorem]{Corollary}
\newtheorem{remark}[theorem]{Remark}
\theoremstyle{definition}
\newtheorem{definition}[theorem]{Definition}

\newcommand{\C}{\mathbb C}
\newcommand{\R}{\mathbb R}
\newcommand{\Pj}{\mathbb P}
\newcommand{\HH}{\mathcal H}
\newcommand{\Gr}{\operatorname{Gr}_{\C}}

\title{A Chow-Type theorem from asymptotic directions}
\author{Nhan Nguyen}
\address{FPT University, Danang, Vietnam}
\email{nguyenxuanvietnhan@gmail.com}

\subjclass[2020]{32B15, 28A78, 14P10, 03C64}
\keywords{Complex analytic set; algebraic set; tangent cone at infinity; asymptotic direction; Hausdorff measure; Chow theorem; o-minimal structure}

\begin{document}

\maketitle

\begin{abstract}
Let $X\subset \C^n$ be a closed pure $d$-dimensional complex analytic set. We associate to $X$ its set of projective asymptotic directions
$$
 \Sigma_\infty(X)
:=\{\ell\in \Pj^{n-1}(\C):\ell\cap  C_\infty(X)\ne\{0\}\},
$$
where $ C_\infty(X)$ is the total tangent cone at infinity. We prove the metric Chow-type characterization
$$
X\text{ is algebraic}
\quad\Longleftrightarrow\quad
\HH^{2d}\bigl( \Sigma_\infty(X)\bigr)=0.
$$
As an application, if $C_\infty(X)$ is definable in an o-minimal expansion
of $(\mathbb R,+,\cdot)$ and $
\dim_{\mathbb R} C_\infty(X)\le 2d$,
then $X$ is algebraic.
\end{abstract}

\section{Introduction}
Chow's theorem  is one of the fundamental links between complex analytic and
algebraic geometry: every closed complex analytic subset of complex
projective space is algebraic \cite{Chow49, Chirka89}. In affine space, the situation is quite
different. Closed complex analytic subsets of $\C^n$ may be transcendental,
and algebraicity can only be recovered by imposing some global restriction
on their behavior at infinity.

Several classical results express this idea in different ways. Rudin
characterized algebraic analytic sets by polynomial growth with respect to a
suitable linear projection \cite{Rudin67}. Bishop and Stoll obtained a characterization
in terms of the optimal Euclidean volume growth $O(r^{2d})$ for a pure
$d$-dimensional analytic set \cite{Bishop64, Stoll, Chirka89}.

Tangent cones at infinity provide another natural description of asymptotic geometry. For complex algebraic sets, Le and Pham \cite{Le-Pham18} proved that the geometric tangent cone at infinity coincides with the algebraic cone defined by the highest-degree terms of defining equations. Dias and Ribeiro \cite{Dias22}, and latter Sampaio \cite{Sampaio23} obtained Chow-type results
in which algebraicity is detected by algebraic tangent cone at infinity.

A different source of affine algebraicity criteria comes from tame geometry.
Peterzil and Starchenko \cite{Peterzil09} proved that a closed complex analytic subset of
$\C^n$ which is definable in an o-minimal structure is
algebraic. Brosnan \cite{Brosnan23} later gave a short proof  by
combining the Bishop--Stoll criterion  with polynomial volume estimates for
definable sets. 

This paper introduces a much weaker condition at infinity. Instead of requiring the analytic set, or even its tangent cone to be algebraic or definable, we look only at the set of directions at infinity $\Sigma_\infty(X)$. Our main result (Theorem \ref{thm_metric-chow}), together with the algebraic case
recalled in Remark \ref{rem:algebraic-cone}, yields the characterization
$$
X\text{ is algebraic}
\quad\Longleftrightarrow\quad \HH^{2d}(\Sigma_\infty(X))=0
$$
(see Theorem \ref{thm_characterization}).
In particular, for $d > 0$, there is a clear size gap: if $X$ is algebraic, $\Sigma_\infty(X)$ has Hausdorff dimension at most $2d-2$, but if $X$ is nonalgebraic, its dimension is at least $2d$ (Corollary \ref{cor:dimension-gap}).

This intuition becomes much clearer when we look at complementary linear spaces. As shown in Corollary \ref{cor_intersection-all-direction1}, if a pure $d$-dimensional analytic
set is nonalgebraic, then every complex $(n-d)$-plane contains a nonzero
tangent vector at infinity. Equivalently, its projective direction set meets
every projective $(n-d-1)$-plane. For hypersurfaces this means that a
nonalgebraic analytic hypersurface has every projective direction at infinity (Corollary \ref{cor_intersection-all-direction2}). Consequently, the graph of a nonpolynomial entire function
$f:\C^d\to\C$ has all directions of $\mathbb{P}^d(\C)$ at infinity.

The proof is based on two simple ingredients. First, a Hausdorff-measure
estimate on the complex Grassmannian shows that if the projective direction
set is sufficiently small, then there exists a complementary complex plane
which avoids all nonzero tangent vectors at infinity. Second, such an
avoiding plane forces a linear growth estimate for the analytic set.
Rudin's geometric criterion then gives algebraicity. 

The o-minimal result which originally motivated this work becomes an
immediate consequence of the metric criterion. If the tangent cone at
infinity is definable and has real dimension at most $2d$, then its
projective direction set has Hausdorff dimension at most $2d-1$, and hence
the analytic set is algebraic (Theorem \ref{thm:definable-cone}). This extends earlier results of Dias--Ribeiro \cite{Dias22} and Sampaio \cite{Sampaio23} by assuming $C_\infty(X)$ is definable rather than algebraic. It also strengthens Peterzil--Starchenko’s affine definable Chow theorem \cite{Peterzil09} : instead of requiring $X$ to be definable, we only need $C_\infty(X)$ to be definable.

The paper is organized as follows. Section \ref{sec2} develops the geometry of
asymptotic directions and proves the main algebraicity criterion.
Section \ref{sec3} contains the geometric and o-minimal applications.

Throughout the paper, we identify $\C^n\equiv\R^{2n}$.  We write $S^{2n-1}$ for the
unit sphere in $\C^n$, $\Pj^{n-1}(\C)$ for the space of complex lines through the origin in $\C^n$, and $\Gr(k,n)$
for the Grassmannian of $k$-dimensional complex linear subspaces of
$\C^n$. We denote by $\HH^s$ the $s$-dimensional Hausdorff measure,
by $\dim_H$ Hausdorff dimension, and by $\dim_{\R}$ and $\dim_{\C}$ real and complex dimensions, respectively. Hausdorff measures on compact smooth manifolds are taken with respect to fixed smooth Riemannian metrics. Since any two such metrics are equivalent on a compact manifold, the property of having zero $s$-dimensional Hausdorff measure is metric independent.

\section{Asymptotic directions and algebraicity}\label{sec2}
\subsection{Asymptotic directions}

\begin{definition}\label{def:tangent-infinity}\rm
Let $A\subset\C^n$. The \emph{total tangent cone at infinity} of $A$ is
$$
 C_\infty(A)
:=
\{
 v\in\C^n:
 \exists\,x_j\in A,\ t_j>0,\ \|x_j\|\to\infty,\ t_j\to\infty, \, x_j/t_j\to v\}$$
Its \emph{spherical asymptotic direction set} is
$$
 D_\infty(A):= C_\infty(A)\cap S^{2n-1}.
$$
The
\emph{projective asymptotic direction set} of $A$ is
$$
 \Sigma_\infty(A)
:=
\{\ell\in\Pj^{n-1}(\C):\ell\cap C_\infty(A)\ne\{0\}\}.
$$
Equivalently, if
$$
\phi:\C^n\setminus\{0\}\longrightarrow\Pj^{n-1}(\C),
\qquad
\phi (v)=\C v,
$$
is the canonical projection, then
$$
 \Sigma_\infty(A)=\phi\bigl( C_\infty(A)\setminus\{0\}\bigr) = \phi(D_\infty(A)).
$$
For a nonzero complex linear subspace $L\subset\C^n$, we write
$$
\Pj(L):=\{\ell\in\Pj^{n-1}(\C):\ell\subset L\}.
$$
Thus, if $\dim_{\C}L=k$, then $\Pj(L)\simeq\Pj^{k-1}(\C)$.
\end{definition}

It is immediately follows from definition that: 
\begin{lemma}\label{lem:projective-affine-bridge}
Let $A\subset\C^n$ and let $L\subset\C^n$ be a nonzero complex linear subspace. Then
$
\Pj(L)\cap \Sigma_\infty(A)\ne\varnothing
\quad\Longleftrightarrow\quad
L\cap C_\infty(A)\ne\{0\}.
$
\end{lemma}

\begin{remark}\label{rem:algebraic-cone}\rm
If $X\subset\C^n$ is algebraic, then the tangent cone at infinity
defined above coincides with the algebraic tangent cone at infinity;
see \cite{Le-Pham18}. In particular,
$C_\infty(X)$ is a complex algebraic cone. If $X$ is pure
$d$-dimensional and $d>0$, then
\[
\dim_{\C}C_\infty(X)=d,
\]
and therefore $\Sigma_\infty(X)$ is a complex projective algebraic
set of dimension $d-1$.
\end{remark}

The following elementary observation.

\begin{lemma}\label{lem_spherical-directions}
Let $A\subset\C^n$ be unbounded. Then
$$
 D_\infty(A)
=
\left\{
 u\in S^{2n-1}:
 \exists\,x_j\in A,\ \|x_j\|\to\infty,
 \frac{x_j}{\|x_j\|}\to u
\right\}.
$$
Moreover,
$$
 C_\infty(A)=\R_{\ge0}\, D_\infty(A).
$$
In particular, $ D_\infty(A)$ is a nonempty compact subset of the sphere and $ \Sigma_\infty(A)$ is compact.
\end{lemma}

\begin{proof}
If $u\in D_\infty(A)$, choose $x_j\in A$ and $t_j\to\infty$ with $x_j/t_j\to u$ and $\|u\|=1$. Then
$$
\frac{\|x_j\|}{t_j}\longrightarrow1,
\qquad
\frac{x_j}{\|x_j\|}
=
\frac{x_j/t_j}{\|x_j\|/t_j}
\longrightarrow u.
$$
Conversely, if $x_j/\|x_j\|\to u$, take $t_j=\|x_j\|$.

The same argument, with $t_j=\|x_j\|/r$, shows that $ru\in C_\infty(A)$ for every $r>0$ and $u\in D_\infty(A)$. Also $0\in C_\infty(A)$ by taking $t_j=\|x_j\|^2$. This gives the cone formula. 

Finally, since $ D_\infty(A)$ is  the set of limit points on the compact sphere of normalized points of $A$ tending to infinity, hence it is nonempty and compact. Its image under the continuous map $u\mapsto\C u$ from $S^{2n-1}$ to $\Pj^{n-1}(\C)$ is therefore compact. This image is precisely $ \Sigma_\infty(A)$.
\end{proof}

\subsection{Avoiding planes and algebraicity}
Let $1\le d<n$ and put $$
G:=\Gr(n-d,n).$$ Since $\dim_{\C}G=d(n-d)$,   $\dim_{\R}G=2d(n-d)$. For a subset $E\subset\Pj^{n-1}(\C)$ define
$$
\mathcal B(E)
:=\{L\in G:\Pj(L)\cap E\ne\varnothing\}.
$$

\begin{proposition}\label{prop:incidence}
Let $E\subset\Pj^{n-1}(\C)$. If
$\HH^{2d}(E)=0$
then
$\HH^{2d(n-d)}\bigl(\mathcal B(E)\bigr)=0$.
In particular, almost every $L\in\Gr(n-d,n)$ satisfies $
\Pj(L)\cap E=\varnothing$.
\end{proposition}

\begin{proof}
Consider the compact complex manifold
$$
\mathcal I
:=
\{(\ell,L)\in\Pj^{n-1}(\C)\times G:\ell\subset L\},
$$
with projections
$$
q:\mathcal I\to\Pj^{n-1}(\C),
\qquad
\pi:\mathcal I\to G.
$$
For a fixed complex line $\ell\in\Pj^{n-1}(\C)$, the fiber $q^{-1}(\ell)$ consists of the complex $(n-d)$-planes containing $\ell$. Passing to the quotient by $\ell$ gives a natural identification
$$
q^{-1}(\ell)
\simeq
\Gr(n-d-1,n-1).
$$
The projection $q$ is the Grassmannian bundle associated with the quotient bundle over $\Pj^{n-1}(\C)$. In particular, $q$ is a smooth fiber bundle with compact fiber
$$
F:=\Gr(n-d-1,n-1),
\qquad
\dim_{\R}F=m:=2d(n-d-1).
$$
Set
$$
\mathcal I_E:=q^{-1}(E).$$
Since the base is compact, choose finitely many bundle-trivializing open sets $U_i$ and relatively compact coordinate neighborhoods $V_i \Subset U_i$ whose union covers $\Pj^{n-1}(\C)$. Let
$$
\Phi_i:q^{-1}(U_i)\longrightarrow U_i\times F
$$
be a smooth bundle trivialization. On the compact sets $q^{-1}(\overline{V_i})$ and $\overline{V_i}\times F$, both $\Phi_i$ and $\Phi_i^{-1}$ have uniformly bounded differentials. Hence, the restricted maps
$$
\Phi_i:q^{-1}(\overline{V_i})\longrightarrow\overline{V_i}\times F
$$
are bi-Lipschitz. Consequently, $
q^{-1}(E\cap V_i)$
is bi-Lipschitz equivalent to
$(E\cap V_i)\times F$.
Since $F$ is a compact $m$-dimensional $C^1$ manifold and $\HH^{2d}(E) = 0$, $\HH^{2d+m}\bigl((E\cap V_i)\times F\bigr)=0$, and hence $\HH^{2d+m} (q^{-1}(V_i \cap E))=0$ (because bi-Lipschitz maps preserve Hausdorff-null sets in a fixed dimension; see, for example \cite{Federer69}). Since the $\{V_i\cap E\}$ is a finite cover of $E$, it follows that
$\HH^{2d+m}(\mathcal I_E)=0$.
Note that
$$
2d+m
=2d+2d(n-d-1)
=2d(n-d)
=\dim_{\R}G.
$$
We then have $\HH^{2d(n-d)}(\mathcal I_E)=0$.

Finally, the projection $\pi:\mathcal I\to G$ is smooth between compact Riemannian manifolds and hence Lipschitz. Therefore $\HH^{2d(n-d)}\bigl(\pi(\mathcal I_E)\bigr)=0$.
Since
$$
\pi(\mathcal I_E)
=
\{L\in G:\Pj(L)\cap E\ne\varnothing\}
=
\mathcal B(E),
$$
the proof is complete.
\end{proof}

Applying Proposition \ref{prop:incidence} to $E= \Sigma_\infty(A)$ gives:

\begin{corollary}\label{cor:avoiding-plane}
Let $A\subset\C^n$ be unbounded and let $1\le d<n$. If $\HH^{2d}\bigl( \Sigma_\infty(A)\bigr)=0$,
then for almost every $L\in\Gr(n-d,n)$, $L\cap C_\infty(A)=\{0\}$.
\end{corollary}

\begin{proposition}\label{prop:linear-growth}
Let $A\subset\C^n$ and let $L\subset\C^n$ be a complex linear subspace such that $
L\cap C_\infty(A)\subset\{0\}$.
Let $V$ be any complex linear complement of $L$ with  $\C^n=V\oplus L$.
Then, there exists $M>0$ such that every $z=x+y\in A$, with $x\in V$ and $y\in L$, satisfies
$$
\|y\|\le M(1+\|x\|).
$$
\end{proposition}

\begin{proof}
Suppose no such $M$ exists. Then, there is a sequence
$
z_j=x_j+y_j\in A
$
such that
$\frac{\|y_j\|}{1+\|x_j\|}\longrightarrow+\infty$.
In particular, $\|y_j\|\to\infty$ and
$
\frac{\|x_j\|}{\|y_j\|}\longrightarrow 0$.
Passing to a subsequence,
$$
\frac{y_j}{\|y_j\|}\longrightarrow u\in L,
\qquad \|u\|=1.
$$
Set $
t_j:=\|y_j\|$.
Then, $t_j\to\infty$ and
$$
\frac{z_j}{t_j}
=
\frac{x_j}{\|y_j\|}
+
\frac{y_j}{\|y_j\|}
\longrightarrow u.
$$
Thus
$u\in L\cap C_\infty(A)$,
contradicting the hypothesis.
\end{proof}

We recall Rudin's geometric criterion.

\begin{definition}\label{def:algebraic-region}
An \emph{algebraic region of type $(d,n)$} is a set of the form
$$
\Omega
=
\left\{
 x+y:
 \begin{array}{l}
 x\in V,\ y\in L,\ \C^n=V\oplus L,\\
 \dim_{\C}V=d,\ \dim_{\C}L=n-d,\\
 \|y\|\le A(1+\|x\|)^B
 \end{array}
\right\},
$$
for some complex linear subspaces $V, L$ and some constants $A,B>0$.
\end{definition}
Rudin's theorem states that 
\begin{theorem}[{Rudin \cite{Rudin67}, \cite[Theorem 3, p 78]{Chirka89}}]\label{thm_Rudin} A closed pure $d$-dimensional complex analytic subset of $\C^n$ is algebraic if and only if it is contained in an algebraic region of type $(d,n)$
\end{theorem}
.

\begin{theorem}[Avoiding-plane criterion]\label{thm_intrinsic}
Let $X\subset\C^n$ be a closed pure $d$-dimensional complex analytic set. Assume there exists a complex linear subspace
$$
L\subset\C^n,
\qquad
\dim_{\C}L=n-d,
$$
such that
$$
L\cap C_\infty(X)\subset\{0\}.
$$
Then, $X$ is algebraic. More precisely, $X$ is contained in an algebraic region of type $(d,n)$ with exponent $B=1$.
\end{theorem}

\begin{proof}
Choose a complex linear complement $V$ of $L$. Proposition \ref{prop:linear-growth} gives $M>0$ such that
$$
X\subset
\{x+y:x\in V,\ y\in L,\ \|y\|\le M(1+\|x\|)\}.
$$
This is an algebraic region of type $(d,n)$ with exponent $1$. Rudin's criterion (Theorem \ref{thm_Rudin}) implies that $X$ is algebraic.
\end{proof}

\begin{theorem}[Metric Chow theorem]\label{thm_metric-chow}
Let $X\subset\C^n$ be a closed pure $d$-dimensional complex analytic set. If
$$
\HH^{2d}\bigl( \Sigma_\infty(X)\bigr)=0,
$$
then $X$ is algebraic.

If $1\le d<n$, then 
$$
\HH^{2d(n-d)}
\bigl(
\{L\in\Gr(n-d,n):L\cap C_\infty(X)\ne\{0\}\}
\bigr)=0.
$$
\end{theorem}

\begin{proof}
Assume first $1\le d<n$. Then $X$ is unbounded. By Corollary \ref{cor:avoiding-plane}, almost every $L\in\Gr(n-d,n)$ satisfies $
L\cap C_\infty(X)=\{0\}
$. Theorem \ref{thm_intrinsic} then implies that $X$ is algebraic. The second conclusion is  Proposition \ref{prop:incidence} applied to $E= \Sigma_\infty(X)$.

If $d=n$, $X = \C^n$, hence algebraic.

Finally, let $d=0$. Since $\HH^0( \Sigma_\infty(X))=0$, the set $ \Sigma_\infty(X)$ is empty. If $X$ were unbounded, Lemma \ref{lem_spherical-directions} would give a nonempty asymptotic direction set, a contradiction. Thus $X$ is bounded. Since $X$ is closed, it is compact. A zero-dimensional analytic subset of $\C^n$ is discrete, hence the
compact set $X$ is finite.
\end{proof}

\section{Applications}\label{sec3}

\subsection{Geometric consequences}

\begin{theorem}[Metric characterization]\label{thm_characterization}
Let $X\subset\C^n$ be a closed pure $d$-dimensional complex analytic set. The following are equivalent:
\begin{enumerate}[label=\textup{(\roman*)}]
\item $X$ is algebraic;
\item $\HH^{2d}\bigl( \Sigma_\infty(X)\bigr)=0$;
\item $\dim_H \Sigma_\infty(X)<2d$.
\end{enumerate}
Moreover, if $d>0$ and $X$ is algebraic, then
$$
\dim_H \Sigma_\infty(X)=2d-2.
$$
\end{theorem}

\begin{proof}
If $X$ is algebraic and $d>0$, Remark \ref{rem:algebraic-cone} gives that $ \Sigma_\infty(X)$ is a complex projective algebraic set of complex dimension $d-1$. Hence, $
\dim_H \Sigma_\infty(X)=2d-2<2d$. Then,  (i) implies (iii), and (iii) implies (ii). The implication (ii) implies (i) is Theorem \ref{thm_metric-chow}.

If $d=0$, an algebraic $X$ is finite and therefore has no asymptotic directions. Theorem \ref{thm_metric-chow} gives the reverse implication. Thus the equivalence also holds in dimension zero.
\end{proof}

\begin{corollary}\label{cor:dimension-gap}
Let $X\subset\C^n$ be a closed pure $d$-dimensional complex analytic set with $d>0$. Then exactly one of the following alternatives occurs:
\begin{enumerate}[label=\textup{(\roman*)}]
\item $X$ is algebraic and
$$
\dim_H \Sigma_\infty(X)=2d-2;
$$
\item $X$ is nonalgebraic and
$$
\HH^{2d}\bigl( \Sigma_\infty(X)\bigr)>0,
\qquad
\dim_H \Sigma_\infty(X)\ge2d.
$$
\end{enumerate}
In particular, the value $2d-1$ cannot occur as $\dim_H \Sigma_\infty(X)$.
\end{corollary}

\begin{proof} (i) follows from Theorem \ref{thm_characterization}. Assume now that $X$ is nonalgebraic. Theorem \ref{thm_characterization} implies $
\HH^{2d}( \Sigma_\infty(X))>0$. Thus, $\dim_H \Sigma_\infty(X)\ge2d$. This proves (ii). 
\end{proof}

\begin{corollary}\label{cor_intersection-all-direction1}
Let $0<d<n$ and let $X\subset\C^n$ be a closed pure $d$-dimensional complex analytic set. If $X$ is nonalgebraic, then
$$
\Pj(L)\cap \Sigma_\infty(X)\ne\varnothing
$$
for every $L\in\Gr(n-d,n)$.
\end{corollary}

\begin{proof}
If some $L$ satisfied $\Pj(L)\cap \Sigma_\infty(X)=\varnothing$, then Lemma \ref{lem:projective-affine-bridge} would give $L\cap C_\infty(X)=\{0\}$.
Theorem \ref{thm_intrinsic}  then imply that $X$ is algebraic.
\end{proof}

Corollary \ref{cor_intersection-all-direction1} applied to $d = n-1$ gives: 

\begin{corollary}\label{cor_intersection-all-direction2}
Let $X\subset\C^n$ be a closed pure $(n-1)$-dimensional complex analytic set. Then either $X$ is algebraic or
$$
 \Sigma_\infty(X)=\Pj^{n-1}(\C).
$$
Equivalently,
$$
 \Sigma_\infty(X)\ne\Pj^{n-1}(\C)
\quad\Longrightarrow\quad
X\text{ is algebraic}.
$$
\end{corollary}

Now we consider an entire complex analytic map $
F=(f_1,\dots,f_m):\C^d\to\C^m$.
Let $$
\Gamma_F
:=\{(z,F(z)):z\in\C^d\}\subset\C^{d+m}
$$
be its graph. This is a smooth closed complex analytic set of dimension $d$.

\begin{corollary}\label{cor_entire-map}
Let $F:\C^d\to\C^m$ be entire. The following are equivalent:
\begin{enumerate}[label=\textup{(\roman*)}]
\item $F$ is polynomial;
\item $\HH^{2d}\bigl( \Sigma_\infty(\Gamma_F)\bigr)=0$;
\item $\dim_H \Sigma_\infty(\Gamma_F)<2d$.
\end{enumerate}
If $F$ is nonpolynomial, then $ \Sigma_\infty(\Gamma_F)$ meets every projective $(m-1)$-plane in $\Pj^{d+m-1}(\C)$ and
$$
\HH^{2d}\bigl( \Sigma_\infty(\Gamma_F)\bigr)>0.
$$
In particular, for an entire function $f:\C^d\to\C$,
$$
f\text{ nonpolynomial}
\quad\Longrightarrow\quad
 \Sigma_\infty(\Gamma_f)=\Pj^d(\C).
$$
\end{corollary}

\begin{proof}
If $F$ is polynomial, then $\Gamma_F$ is algebraic, so Theorem \ref{thm_characterization} gives (ii) and (iii). Conversely, if (ii) or (iii) holds, Theorem \ref{thm_characterization} implies that $\Gamma_F$ is algebraic.

Since $\Gamma_F$ is irreducible and the first coordinate projection
 $
\pi:\Gamma_F\to\C^d $
is generically one-to-one, it is birational. Hence each coordinate $f_j$ represents a rational function on $\C^d$. Since $f_j$ is entire, it has no pole on $\C^d$, so it is polynomial. This proves the equivalence. The remaining statements follow from Corollaries  \ref{cor_intersection-all-direction1}, \ref{cor:dimension-gap}  and \ref{cor_intersection-all-direction2}.
\end{proof}

\subsection{Tame tangent-cones criterion}\label{sec:tame}

In this section, all definability is with respect to a fixed o-minimal structure on $\R$. We use a standard fact about o-minimal dimension: $\dim_{\R} C_\infty(X) \le  \dim_{\R} (X, \infty)$, where $(X, \infty)$ is the germ of the definable set $X$ at infinity. The proof is straightforward by arguments similar to \cite[Lemma 1.2]{Kurdyka89} (see also \cite[Corollary 2.18]{Fernandes20}). For an introduction to o-minimal structures, we refer the reader to  \cite{ Coste00, Dries98, Dries-Miller96, Loi1}.

\begin{theorem}[Definable tangent-cone criterion]\label{thm:definable-cone}
Let $X\subset\C^n$ be a closed pure $d$-dimensional complex analytic set. Suppose that $ C_\infty(X)$ is definable and $
\dim_{\R} C_\infty(X)\le2d$.
Then, $X$ is algebraic.
\end{theorem}

\begin{proof}
The cases $d=0$ and $d=n$ are immediate from the arguments in Theorem \ref{thm_metric-chow}, so we  assume $1\le d<n$. Then, $X$ is unbounded. Since $C_\infty (X)$ is definable and
$\dim_{\R} C_\infty(X) \leq  2d$, $D_\infty(X)$ is  definable and $\dim_{\R} D_\infty(X) \leq 2d-1$. 
The projection $
\phi: S^{2n-1}\to\Pj^{n-1}(\C),\; u \mapsto \mathbb{C} u$
is semialgebraic, so $ \Sigma_\infty(X) = \phi(D_\infty(X))$ is definable and
$
\dim_{\R} \Sigma_\infty(X)\le2d-1$.
Therefore
$\HH^{2d}\bigl( \Sigma_\infty(X)\bigr)=0
$;
and Theorem \ref{thm_metric-chow} implies that $X$ is algebraic.
\end{proof}

\begin{corollary}\label{cor:def-gap}
Let $X\subset\C^n$ be a closed pure $d$-dimensional complex analytic set with $d>0$ and assume $ C_\infty(X)$ is definable. Then
$$
X\text{ algebraic}
\quad\Longleftrightarrow\quad
\dim_{\R} C_\infty(X)\le2d
\quad\Longleftrightarrow\quad
\dim_{\R} C_\infty(X)=2d.
$$
Consequently, if $X$ is nonalgebraic, then
$$
\dim_{\R} C_\infty(X)\ge2d+1.
$$
\end{corollary}

\begin{proof}
The implication from the dimension bound to algebraicity is Theorem \ref{thm:definable-cone}. The remaining implications are obvious.
\end{proof}

\begin{corollary}\label{cor:def-projective-gap}
Assume $d>0$ and $ \Sigma_\infty(X)$ is definable. Then
$$
X\text{ algebraic}
\quad\Longrightarrow\quad
\dim_{\R} \Sigma_\infty(X)=2d-2,
$$
whereas
$$
X\text{ nonalgebraic}
\quad\Longrightarrow\quad
\dim_{\R} \Sigma_\infty(X)\ge2d.
$$
\end{corollary}

\begin{proof}
The algebraic case follows from Theorem \ref{thm_characterization}. If $X$ is nonalgebraic, Corollary \ref{cor:dimension-gap} gives
$\HH^{2d}( \Sigma_\infty(X))>0$.
Hence, $\dim_{\R} \Sigma_\infty(X)\ge2d$.
\end{proof}

We now recover the affine definable Chow theorem of Peterzil--Starchenko \cite[Corollary 4.5]{Peterzil09}) by looking only at the tangent cone at infinity.

\begin{corollary}[Definable Chow Theorem]\label{cor_def-chow}
Let $X\subset\C^n$ be a closed pure-dimensional complex analytic set. If $X$ is definable, then $X$ is algebraic.
\end{corollary}

\begin{proof}
Let $d=\dim_{\C}X$. If $d=0$, o-minimality implies that $X$ is finite. Assume $d>0$. 
Since $X$ is definable, $C_\infty(X)$ is definable and $\dim_{\mathbb R} C_\infty(X) \leq \dim_{\mathbb R} X = 2d$. The result immediately follows from Theorem \ref{thm:definable-cone}. 
\end{proof}

\bibliographystyle{abbrv}
\bibliography{References}

\end{document}